\documentclass[oneside,english]{amsart}
\usepackage{mathpazo}
\usepackage[T1]{fontenc}
\usepackage[latin9]{inputenc}
\usepackage{geometry}
\usepackage{babel}
\usepackage{textcomp}
\usepackage{amstext}
\usepackage{amsthm}
\usepackage{amssymb}
\usepackage[numbers]{natbib}
\usepackage[unicode=true,pdfusetitle,
 bookmarks=true,bookmarksnumbered=false,bookmarksopen=false,
 breaklinks=false,pdfborder={0 0 1},backref=false,colorlinks=false]
 {hyperref}

\makeatletter
\numberwithin{equation}{section}
\numberwithin{figure}{section}
\theoremstyle{plain}
\newtheorem{thm}{\protect\theoremname}
\theoremstyle{plain}
\newtheorem{lem}[thm]{\protect\lemmaname}

\makeatother

\providecommand{\lemmaname}{Lemma}
\providecommand{\theoremname}{Theorem}

\begin{document}
\title{Brouwer's conjecture holds asymptotically almost surely}
\author{Israel Rocha}
\address{The Czech Academy of Sciences, Institute of Computer Science, Pod
Vodárenskou v\v{e}ží 2, 182~07 Prague, Czech Republic. With institutional
support RVO:67985807.}
\thanks{\emph{Rocha} was supported by the Czech Science Foundation, grant
number GJ16-07822Y}
\email{israelrocha@gmail.com}
\begin{abstract}
We show that for a sequence of random graphs Brouwer's conjecture
holds true with probability tending to one as the number of vertices
tends to infinity. Surprisingly, it was found that a similar statement
holds true for weighted graphs with possible negative weights as well.
For graphs with a fixed number of vertices, the result implies that
there are constants $C>0$ and $n_{0}$ such that if $n\geq n_{0}$
then among all $2^{{n \choose 2}}$ graphs with $n$ vertices, at
least $\left(1-\exp\left(-Cn\right)\right)2^{{n \choose 2}}$ graphs
satisfy Brouwer's conjecture.
\end{abstract}

\maketitle

\section{Introduction}

Brouwer's conjecture \citep{BrouwerHammers} states that any graph
$G=\left(E,V\right)$ with Laplacian matrix $L$ and eigenvalues $\lambda_{1}\geq\ldots\geq\lambda_{n}$
satisfies
\begin{equation}
\sum_{i=1}^{k}\lambda_{i}\leq\left|E\right|+{k+1 \choose 2},\label{eq:browersConjecture}
\end{equation}
for $k=1,\ldots,n$. There have been many partial progresses on this
conjecture using particular methods from matrix theory. This conjecture
seems to be a difficult problem and to this date it remains open.
In this paper we present an approach to this problem using methods
from random matrices and random graphs. Here we address the following
question: for how many graphs inequality (\ref{eq:browersConjecture})
holds? In contrast to what previous investigations have been focused
so far, instead of approaching the problem for graphs enjoying a prescribed
structure, we show that a great proportion of graphs satisfy the conjecture. 

In this paper it is shown that Brouwer's conjecture holds asymptotically
almost surely for random graphs under general conditions to be specified
later. That is to say that graphs that potentially do not satisfy
the inequality in Brouwer's conjecture are rare, in a precise measure
theoretic sense. That suggests a change of focus in the research that
has been done on this problem so far. Instead of searching for new
families of graphs for which the conjecture holds, one should attempt
to understand these rare cases for which the conjecture potentially
do not hold.

It comes as a surprise that this approach reveals that a similar statement
holds true for weighted graphs, and even with possible negative weights.
To state the result precisely some notation is needed. First, we denote
by $w_{uv}$ the weight of an edge $uv\in E$ and set $w_{uv}=0$
in case $uv\notin E$. Then, we define $e(G)=\sum_{uv\in E}w_{uv}$.
The Laplacian matrix of a weighted graph have the number $-w_{uv}$
in the off-diagonal entry $uv$ and $\sum_{j\neq u}w_{uj}$ in the
diagonal entry $uu$. As usual, an unweighted graph can be seen as
a $\{0,1\}$ weighted graph with its standard Laplacian matrix. Clearly,
in this case $\left|E\right|=e(G)$. We show that for a weighted graph
$G$ and for $k=1,\ldots,n$ we have 
\[
\sum_{i=1}^{k}\lambda_{i}\leq e(G)+{k+1 \choose 2},
\]
asymptotically almost surely.

A random weighted graph $G_{n}$ is a graph with $n$ vertices and
edge weights given by a random variable $\xi_{ij}^{(n)}$ for each
$ij\in[n]$. Throughout this paper a sequence of random weighted graphs
is under the general condition described as follow:

(Condition $1$) Let $\left\{ G_{1},G_{2},G_{3},\ldots\right\} $
be a sequence of random weighted graph with $n$ vertices and Laplacian
matrix given by
\[
L_{n}=\left[\begin{array}{cccc}
\sum_{j\neq1}\xi_{1j}^{(n)} & -\xi_{12}^{(n)} & \cdots & -\xi_{1n}^{(n)}\\
-\xi_{21}^{(n)} & \sum_{j\neq2}\xi_{2j}^{(n)} &  & -\xi_{2n}^{(n)}\\
\vdots &  & \ddots & \vdots\\
-\xi_{n1}^{(n)} & \cdots &  & \sum_{j\neq n}\xi_{nj}^{(n)}
\end{array}\right],
\]
where for $i<j$ we have that $\xi_{ij}^{(n)}$ are bounded random
variables on the same probability space and independent for each $n$
(not necessarily identically distributed) with $\xi_{ij}^{(n)}=\xi_{ji}^{(n)}$,
$\mathbb{E}\left[\xi_{ij}^{(n)}\right]=\mu_{n}$, $\text{Var}\left[\xi_{ij}^{(n)}\right]=\sigma_{n}^{2}$
, and
\[
\sup_{i,j,n}\mathbb{E}\left[\left|\left(\xi_{ij}^{(n)}-\mu_{n}\right)/\sigma_{n}\right|^{p}\right]<\infty
\]
for some $p>6$.

We have the setup to state the main result of this paper which implies
Brouwer's conjecture asymptotically almost surely.
\begin{thm}
\label{thm:mainThm}Assume $\left\{ G_{1},G_{2},G_{3},\ldots\right\} $
are independent random graphs as in Condition $1$. If $\mu_{n}\in(0,1-\gamma]$
for some $\gamma\in(0,1)$ and $\text{\ensuremath{\frac{\mu_{n}}{\sigma_{n}}\left(\frac{n}{\log n}\right)^{1/2}\rightarrow\infty}}$
as $n\rightarrow\infty$, then for $n$ large enough, we have that
\begin{equation}
\mathbb{P}\left[\sum_{i=1}^{k}\lambda_{i}\leq e(G_{n})+{k+1 \choose 2}\right]\geq1-\exp\left(-C\text{\ensuremath{\mu_{n}}}n\right),\label{eq:probBrouwersConj}
\end{equation}
for some constant $C>0$ that depends only on $\gamma$.
\end{thm}

It is interesting to understand what this result is saying for the
set of graphs with fixed number of vertices. For such graphs this
result is implying a strong statement that quantifies the number of
graphs satisfying Brouwer's conjecture. In fact, the most simple instance
of Theorem \ref{thm:mainThm} is the most meaningful for the conjecture.
To see that consider the Erd\H{o}s--Rényi random graph with probability
$1/2$ and distribution $\mathbb{G}(n,1/2)$, i.e., a graph drawn
from this distribution has $n$ vertices where each pair of vertices
has an edge with probability $1/2$ independently at random. It is
a basic fact that any graph with $n$ vertices is equally likely in
the distribution $\mathbb{G}(n,1/2)$. Therefore, Theorem \ref{thm:mainThm}
implies that for $n\geq n_{0}$ among all $2^{{n \choose 2}}$ graphs
with $n$ vertices that exist, at least $\left(1-\exp-Cn\right)2^{{n \choose 2}}$
graphs satisfy Brouwer's conjecture.

The hypothesis $\text{\ensuremath{\frac{\mu_{n}}{\sigma_{n}}\left(\frac{n}{\log n}\right)^{1/2}\rightarrow\infty}}$
in Theorem \ref{thm:mainThm} ensures that $\mu_{n}/\sigma_{n}$ is
not approaching zero too fast. That is necessary because of the concentration
of the largest eigenvalue around $\mu_{n}n$ and because there is
a threshold phenomena happening here. When $\mu_{n}/\sigma_{n}$ goes
to zero fast, the largest eigenvalue concentrates around $\sigma_{n}\sqrt{n\log}n$.
In that case we still can apply the same method, but the difference
in concentration requires a different analysis. For this reason we
provide a separate theorem for this range, where $\mu_{n}/\sigma_{n}$
goes to zero in such a way that $\text{\ensuremath{\frac{\mu_{n}}{\sigma_{n}^{2}}\left(\frac{n}{\log n}\right)\rightarrow\infty}}$.
\begin{thm}
\label{thm:secondMainThm}Assume $\left\{ G_{1},G_{2},G_{3},\ldots\right\} $
are independent random graphs as in Condition $1$. If $\text{\ensuremath{\frac{\mu_{n}}{\sigma_{n}}\left(\frac{n}{\log n}\right)^{1/2}\rightarrow0}}$
and $\sigma_{n}^{2}\frac{\log n}{\mu_{n}n}\rightarrow0$ as $n\rightarrow\infty$,
then for $n$ large enough, we have that
\begin{equation}
\mathbb{P}\left[\sum_{i=1}^{k}\lambda_{i}\leq e(G_{n})+{k+1 \choose 2}\right]\geq1-\exp\left(-C\text{\ensuremath{\mu_{n}}}n\right),\label{eq:probBrouwersConj-1}
\end{equation}
for some constant $C>0$.
\end{thm}

We get into the detailed proof in the next section, where we first
present the method from a general perspective. Our intention is to
give an insight on how such ideas can be used to provide bounds for
the partial sum of eigenvalues of random matrices from different ensembles. 

\section{Idea and proofs}

The main idea is straightforward and it consists in finding functions
$f$ and $g$ depending on $n$ satisfying:
\begin{itemize}
\item There exists $n_{0}$ such that $f(n)\leq g(n)$ for $n\geq n_{0}$
\item $\mathbb{P}\left[\sum_{i=1}^{k}\lambda_{i}\leq f\right]\rightarrow1$
as $n\rightarrow\infty$
\item $\mathbb{P}\left[g\leq e(G_{n})+{k+1 \choose 2}\right]\rightarrow1$
as $n\rightarrow\infty$
\end{itemize}
Once we figure out what $f$ and $g$ should be, by Bonferroni's inequality
we clearly have
\begin{align*}
\mathbb{P}\left[\sum_{i=1}^{k}\lambda_{i}\leq e(G_{n})+{k+1 \choose 2}\right] & \geq\mathbb{P}\left[\sum_{i=1}^{k}\lambda_{i}\leq f\right]+\mathbb{P}\left[g\leq e(G_{n})+{k+1 \choose 2}\right]-1\rightarrow1.
\end{align*}
And that finishes the proof. It is also clear that the technical challenge
here is to find such functions. To do that, we need some information
about the spectrum as $n$ increases. For instance, if the limiting
spectral distribution is known, then we can obtain a candidate for
$f$, which will be an approximation of the limit. For $g$ we can
use some Chernoff-type bound for the random variable that amounts
the total edge weight. 

We remark that there is nothing particular in this idea about the
Laplacian matrix. Apart from the fact that Brouwer's conjecture claims
what the correct bound should be, this method can be applied to any
matrix ensemble for which we know the limiting spectral distribution. 

Next, we proceed with the proof of the main results. Our analysis
relies on the following technical lemma which fully describes the
aforementioned relevant functions $f$ and $g$.
\begin{lem}
\label{lem:epsilondelta}If $\mu_{n}\in(0,1-\gamma]$ for some $\gamma\in(0,1)$,
then there are $\epsilon>0,$ $\text{\ensuremath{\delta>0}},$ and
$n_{0}$ such that for all $n\geq n_{0}$ we have 
\begin{equation}
k(1+\epsilon)\text{\ensuremath{\mu_{n}n}}<\text{\ensuremath{\mu_{n}}}\left(1-\delta\right){n \choose 2}+{k+1 \choose 2}\label{eq:bound1}
\end{equation}
for all $k\in\mathbb{R}$.
\end{lem}

\begin{proof}
First, we define a polynomial in $k$ by $f(k):=\text{\ensuremath{\mu_{n}}}\left(1-\delta\right){n \choose 2}+{k+1 \choose 2}-k(1+\epsilon)\text{\ensuremath{\mu_{n}n}}$.
That has discriminant
\[
\Delta=n^{2}\mu_{n}(\delta-1+\mu_{n}(\epsilon+1)^{2})-n\mu_{n}(\delta+\epsilon)+1/4.
\]
It suffices to find $\epsilon,\text{\ensuremath{\delta}},$ and $n_{0}$
such that the discriminant is negative for all $n\geq n_{0}$. To
this end, we use the upper bound on $\mu_{n}$ to obtain
\[
\Delta\leq n^{2}\mu_{n}(\delta-1+(1-\gamma)(\epsilon+1)^{2})-n\mu_{n}(\delta+\epsilon)+1/4.
\]
Now fix $\delta=\gamma^{2}/2$ and $\epsilon>0$ such that $(\epsilon+1)^{2}=1+\gamma$.
That allow us to bound the discriminant by
\begin{align*}
\Delta & \leq n^{2}\mu_{n}(\delta-1+(1-\gamma)(1+\gamma))-n\mu_{n}(\delta+\epsilon)+1/4\\
 & =n^{2}\mu_{n}(\delta-\gamma^{2})+1/4\\
 & =-n^{2}\mu_{n}\gamma^{2}/2+1/4.
\end{align*}
Clearly, for $n$ large enough the last expression is dominated by
the term $-n^{2}\mu_{n}\gamma^{2}/2$ and therefore there exists a
$n_{0}$ such that for all $n\geq n_{0}$ we have that $f(k)>0$ for
all $k$. That proves the lemma.
\end{proof}
To estimate the largest eigenvalue $\lambda_{\text{max}}\left(L_{n}\right)$
we use Corollary 1.1 (b1) and (b2) from \citep{RandomLapl}. 
\begin{lem}
\label{lem:concentration}Assume $\left\{ G_{1},G_{2},G_{3},\ldots\right\} $
are independent random graphs as in Condition $1$. 
\end{lem}

\begin{enumerate}
\item If~ $\lim_{n\rightarrow\infty}\text{\ensuremath{\frac{\mu_{n}}{\sigma_{n}}\left(\frac{n}{\log n}\right)^{1/2}=\infty}}$
and $\mu_{n}>0$, then $\mathbb{P}\left[\lim_{n\rightarrow\infty}\frac{\lambda_{\text{max}}\left(L_{n}\right)}{n\mu_{n}}=1\right]=1.$
\item If~ $\lim_{n\rightarrow\infty}\text{\ensuremath{\frac{\mu_{n}}{\sigma_{n}}\left(\frac{n}{\log n}\right)^{1/2}=0}}$,
then $\mathbb{P}\left[\limsup_{n\rightarrow\infty}\frac{\lambda_{\text{max}}\left(L_{n}\right)}{\sigma_{n}\sqrt{n\log}n}=2\right]=1.$
\end{enumerate}
Now, we are ready to proceed with the main proof which was roughly
sketched in the beginning of this section.
\begin{proof}[Proof of Theorem \ref{thm:mainThm}]
From now on, we fix $\epsilon,\text{\ensuremath{\delta}},$ and $n_{0}$
given by Lemma \ref{lem:epsilondelta} to obtain that for all $n\geq n_{0}$
we have 
\begin{align}
\mathbb{P}\left[\sum_{i=1}^{k}\lambda_{i}\leq e(G_{n})+{k+1 \choose 2}\right] & \geq\mathbb{P}\left[\sum_{i=1}^{k}\lambda_{i}\leq k\left(1+\epsilon\right)\mu_{n}n\text{ and }\mu_{n}(1-\delta){n \choose 2}\leq e(G_{n})\right]\nonumber \\
 & \geq\mathbb{P}\left[\sum_{i=1}^{k}\lambda_{i}\leq k\left(1+\epsilon\right)\mu_{n}n\right]+\mathbb{P}\left[\mu_{n}(1-\delta){n \choose 2}\leq e(G_{n})\right]-1.\label{eq:ineq1}
\end{align}
By part (1) from Lemma \ref{lem:concentration} 
\[
\mathbb{P}\left[\lim_{n\rightarrow\infty}\frac{\lambda_{\text{max}}\left(L_{n}\right)}{n\mu_{n}}=1\right]=1.
\]
Thus, for $\epsilon$ given by Lemma \ref{lem:epsilondelta} we can
find $m_{0}$ such that $n\geq m_{0}$ implies $\mathbb{P}\left[\lambda_{\text{max}}\left(L_{n}\right)\leq\left(1+\epsilon\right)\mu_{n}n\right]=1.$
That implies
\begin{equation}
\mathbb{P}\left[\sum_{i=1}^{k}\lambda_{i}\leq k\left(1+\epsilon\right)\mu_{n}n\right]=1.\label{eq:sum1}
\end{equation}
Thus, inequality (\ref{eq:sum1}) together with (\ref{eq:ineq1})
gives us that for all $n\geq\min\left\{ m_{0},n_{0}\right\} $ we
have
\begin{equation}
\mathbb{P}\left[\sum_{i=1}^{k}\lambda_{i}\leq e(G_{n})+{k+1 \choose 2}\right]\geq\mathbb{P}\left[\mu_{n}(1-\delta){n \choose 2}\leq e(G_{n})\right].\label{eq:ineq2}
\end{equation}

To bound this probability, we use that the expected number of edges
in $G_{n}$ is $\mu_{n}{n \choose 2}$. We remark that $\left|\xi_{ij}^{(n)}\right|<C$
for all $\xi_{ij}^{(n)}$. Thus, Hoeffding\textasciiacute s inequality
implies that 
\[
\mathbb{P}\left[e(G_{n})\leq(1-\delta)\text{\ensuremath{\mu_{n}}}{n \choose 2}\right]\leq\exp-\frac{\delta^{2}\left(\text{\ensuremath{\mu_{n}}}{n \choose 2}\right)^{2}}{{n \choose 2}C^{2}}=\exp-\frac{\delta^{2}\text{\ensuremath{\mu_{n}^{2}}}{n \choose 2}}{C^{2}}
\]
for all $\delta\in(0,1)$. Equivalently, we have 
\[
\mathbb{P}\left[(1-\delta)\text{\ensuremath{\mu_{n}}}{n \choose 2}\leq e(G_{n})\right]\geq1-\exp-\frac{\delta^{2}\text{\ensuremath{\mu_{n}^{2}}}{n \choose 2}}{C^{2}}.
\]
Finally, this last inequality and inequality (\ref{eq:ineq2}) ensures
that for all $n\geq\min\left\{ m_{0},n_{0}\right\} $ we have

\[
\mathbb{P}\left[\sum_{i=1}^{k}\lambda_{i}\leq e(G_{n})+{k+1 \choose 2}\right]\geq1-\exp\left(-C_{0}\text{\ensuremath{\mu_{n}}}n\right).
\]
for some constant $C_{0}>0$. That finishes the proof.
\end{proof}
Now, the proof of Theorem \ref{thm:secondMainThm} is basically the
same, only the details in the analysis change. For that we need a
different version of Lemma \ref{lem:epsilondelta} given bellow. 
\begin{lem}
\label{lem:approximation2}If $\sigma_{n}^{2}\frac{\log n}{\mu_{n}n}\rightarrow0$
as $n\rightarrow\infty$, then there are $\epsilon>0,$ $\text{\ensuremath{\delta>0}},$
and $n_{0}$ such that for all $n\geq n_{0}$ we have
\begin{equation}
k(2+\epsilon)\text{\ensuremath{\sigma_{n}}}\sqrt{n\log n}<\text{\ensuremath{\mu_{n}}}\left(1-\delta\right){n \choose 2}+{k+1 \choose 2}\label{eq:bound2-1}
\end{equation}
for all $k$.
\end{lem}

\begin{proof}
First, there is a constant $C>0$ such that for $n$ large, we have
\[
\text{\ensuremath{\mu_{n}}}\left(1-\delta\right){n \choose 2}+{k+1 \choose 2}-k(2+\epsilon)\text{\ensuremath{\sigma_{n}}}\sqrt{n\log n}>\text{\ensuremath{C\mu_{n}}}\left(1-\delta\right)n^{2}+\frac{k^{2}}{2}-k(2+\epsilon)\text{\ensuremath{\sigma_{n}}}\sqrt{n\log n}.
\]
We proceed to show that there are $\epsilon>0$ and $\text{\ensuremath{\delta>0}}$,
where $\text{\ensuremath{C\mu_{n}}}\left(1-\delta\right)n^{2}+\frac{k^{2}}{2}-k(2+\epsilon)\text{\ensuremath{\sigma_{n}}}\sqrt{n\log n}>0$
for $n$ large enough and that finishes the proof.

To show it we define a polynomial in $k$ by $f(k):=\frac{k^{2}}{2}-k(2+\epsilon)\text{\ensuremath{\sigma_{n}}}\sqrt{n\log n}+\text{\ensuremath{C\mu_{n}}}\left(1-\delta\right)n^{2}$.
This polynomial has discriminant
\begin{align*}
\Delta & =\left(2+\epsilon\right)^{2}\sigma_{n}^{2}n\log n-2C\mu_{n}\left(1-\delta\right)n^{2}.
\end{align*}
It suffices to find $\epsilon,\text{\ensuremath{\delta}},$ and $n_{0}$
such that the discriminant is negative for all $n\geq n_{0}$. In
fact, it is enough to choose any $\epsilon>0$ and $\text{\ensuremath{\delta\in(0,1)}}$.
Notice that $\Delta<0$ is equivalent to
\[
\sigma_{n}^{2}\frac{\log n}{\mu_{n}n}<\frac{2C\left(1-\delta\right)}{\left(2+\epsilon\right)^{2}},
\]
which is true for $n$ large enough, as required.
\end{proof}
\begin{proof}[Proof of Theorem \ref{thm:secondMainThm}]
First, we fix $\epsilon,\text{\ensuremath{\delta}},$ and $n_{0}$
given by Lemma \ref{lem:approximation2} to obtain that for all $n\geq n_{0}$
we have

\begin{align}
\mathbb{P}\left[\sum_{i=1}^{k}\lambda_{i}\leq e(G_{n})+{k+1 \choose 2}\right] & \geq\mathbb{P}\left[\sum_{i=1}^{k}\lambda_{i}\leq k(2+\epsilon)\text{\ensuremath{\sigma_{n}}}\sqrt{n\log n}\text{ and }\mu_{n}(1-\delta){n \choose 2}\leq e(G_{n})\right]\nonumber \\
 & \geq\mathbb{P}\left[\sum_{i=1}^{k}\lambda_{i}\leq k(2+\epsilon)\text{\ensuremath{\sigma_{n}}}\sqrt{n\log n}\right]+\mathbb{P}\left[\mu_{n}(1-\delta){n \choose 2}\leq e(G_{n})\right]-1.\label{eq:ineq2-1-1}
\end{align}
Part (2) from Lemma \ref{lem:concentration} provides us with
\[
\mathbb{P}\left[\limsup_{n\rightarrow\infty}\frac{\lambda_{\text{max}}\left(L_{n}\right)}{\sigma_{n}\sqrt{n\log}n}=2\right]=1.
\]
Again, for $\epsilon$ given by Lemma \ref{lem:approximation2} we
can find $m_{0}$ such that $n\geq m_{0}$ implies 
\[
\mathbb{P}\left[\lambda_{\text{max}}\left(L_{n}\right)\leq\left(2+\epsilon\right)\sigma_{n}\sqrt{n\log}n\right]=1.
\]
 That gives us
\begin{equation}
\mathbb{P}\left[\sum_{i=1}^{k}\lambda_{i}\leq k\left(2+\epsilon\right)\sigma_{n}\sqrt{n\log}n\right]=1.\label{eq:sum2-1}
\end{equation}
Thus, inequality (\ref{eq:ineq2-1-1}) together with (\ref{eq:sum2-1})
gives us that for all $n\geq\min\left\{ m_{0},n_{0}\right\} $ we
have
\begin{equation}
\mathbb{P}\left[\sum_{i=1}^{k}\lambda_{i}\leq e(G_{n})+{k+1 \choose 2}\right]\geq\mathbb{P}\left[\mu_{n}(1-\delta){n \choose 2}\leq e(G_{n})\right].\label{eq:ineq2-2}
\end{equation}
The rest of the proof is the same as in Theorem \ref{thm:mainThm}
and that finishes the proof.
\end{proof}


\begin{thebibliography}{99}
\bibitem{BrouwerHammers}Brouwer, Andries E., Haemers, Willem H. Spectra
of Graphs. Springer-Verlag, New York, 2012.

\bibitem{RandomLapl}Ding, Xue; Jiang, Tiefeng. Spectral distributions
of adjacency and Laplacian matrices of random graphs. Ann. Appl. Probab.
20 (2010), no. 6, 2086-2117.
\end{thebibliography}
\end{document}